\documentclass[11pt]{article}
\usepackage{a4wide}
\usepackage{xcolor}

\usepackage{parskip}
\usepackage{hyperref}
\usepackage{parskip}
\usepackage{setspace}
\usepackage{graphicx}

\usepackage{tikz} 
\usetikzlibrary{decorations, arrows.meta, positioning} 
\usepackage{graphicx}
\usetikzlibrary{decorations.markings}
\usepackage{tikz}

\usetikzlibrary{calc,arrows.meta,decorations.markings,shapes.geometric}
\definecolor{purple}{RGB}{138,43,226}

\usepackage[margin=1in]{geometry}
\usepackage{amsmath,amsthm,amssymb}
\usepackage{hyperref}

\numberwithin{equation}{section}

\newtheorem{theorem}{Theorem}[section]
\newtheorem{lemma}[theorem]{Lemma}

\theoremstyle{remark}

\begin{document}

\title{Raimi's theorem for manifolds with circle symmetry}
\author{Dung The Tran\thanks{VNU University of Science, Hanoi. Email: tranthedung56@gmail.com}}

\date{}
\maketitle

\begin{abstract}
Raimi's classical theorem establishes a partition of the natural numbers with a remarkable unavoidability property: for every finite coloring of $\mathbb{N}$, there is a color class whose translate meets both parts of the partition in infinitely many points. Recently, Kang, Koh, and the author have extended this phenomenon to the circle group, proving that there exists a measurable partition of the circle such that every finite measurable cover admits a rotation whose image meets each part of the partition in positive measure. This paper shows that this phenomenon extends beyond compact abelian groups to a wide class of non-group geometric surfaces that still exhibit \textit{a hidden one-dimensional symmetry}. Specifically, we establish analogs of Raimi's theorem for three families of surfaces (with their natural surface measures): the unit sphere $\mathbb{S}^{n-1} \subset \mathbb{R}^n$, rotational power surfaces (such as cones and paraboloids), and circular cylindrical surfaces. The common feature is that each of these surfaces carries a natural measure-preserving action of the circle group by rotation in a fixed plane and admits a measurable trivialization as a product $\mathbb{T} \times Y$. This circle-bundle structure allows the measurable Raimi partition on the base circle to be lifted to an unavoidable partition on the manifold. Our approach is unified through a general circle-bundle theorem, which reduces all three geometric cases to verifying suitable equivariance and product disintegration properties of the surface measure.
\end{abstract}




\section{Introduction}
Classical Ramsey theory typically asks which structured subsets must appear in any finite colouring of the natural numbers. Raimi proposed a complementary point of view: he asked which partitions of $\mathbb{N}$ cannot be avoided by any finite colouring of $\mathbb{N}$, even after allowing a shift.

More precisely, given a partition $\mathbb{N}=E_1\cup E_2$, we say that $(E_1,E_2)$ is \emph{unavoidable} if for every finite partition $\mathbb{N}=\bigcup\limits_{j=1}^t F_j$ there exist $j\in\{1,\dots,t\}$ and $k\in\mathbb{N}$ such that both $(F_j+k)\cap E_1$ and $(F_j+k)\cap E_2$ are infinite. This formalises Raimi’s viewpoint on unavoidable partitions under shifts.

A classical theorem of Raimi \cite{Raimi} shows that such partitions do exist: there is a partition $\mathbb{N}=E_1\cup E_2$ with the property that for every finite partition $\mathbb{N}=\bigcup\limits_{j=1}^t F_j$ with $t\in\mathbb{N}$, one can find $j\in\{1,\dots,t\}$ and $k\in\mathbb{N}$ such that both $(F_j+k)\cap E_1$ and $(F_j+k)\cap E_2$ are infinite. Raimi’s original proof used topological methods. Hindman later gave an elementary proof \cite[p.~180, Theorem~11.15]{HM79} and showed that one may take $E_1$ to be the set of natural numbers whose last non-zero digit in the ternary expansion is $1$, and $E_2=\mathbb{N}\setminus E_1$.

Strengthenings of this phenomenon, imposing density conditions on the partition sets or guaranteeing positive densities in the conclusion, were obtained by Hegyv\'ari \cite{NH} and by Bergelson and Weiss \cite{Bergelson2}. More recently, Hegyv\'ari, Pach, and Pham \cite{HPP25} introduced a powerful and flexible framework, combining tools from harmonic analysis, additive combinatorics, and group theory, which yields polynomial and finite-group extensions of Raimi’s theorem and makes its connection to Ramsey theory explicit. Their beautiful construction in the finite-group setting has since been extended successfully to the continuous setting for circles group $\mathbb{T}=\mathbb{R}/\mathbb{Z}$ by Kang, Koh, and the author in \cite{KKT25}, which is stated as follows.
\begin{theorem}[Kang--Koh--Tran {\cite{KKT25}}]\label{thm:KKT-circle}
Let $r,t\in\mathbb{N}$ with $r,t\ge2$. There exists a measurable partition
\[
\mathbb{T}= \bigcup_{i=1}^r E_i
\]
such that for every finite measurable cover
\[
C \subset F_1 \cup \cdots \cup F_t,
\]
there exist an index $m \in \{1,\dots,t\}$ and a rotation $R_{\theta}$ satisfying
\begin{align*}
    \mu_1\big(R_{\theta}(F_m) \cap E_i\big) > 0
\quad\text{for all } 1\le i\le r.
\end{align*}
\end{theorem}
The theorem provides a partition of the circle with the property that every finite measurable cover admits a translate meeting each partition element in positive measure.  This paper shows that this phenomenon extends beyond compact abelian groups to a wide class of non-group geometric surfaces that still exhibit \textit{a hidden one-dimensional symmetry}. This answers a question raised in \cite{KKT25} concerning the extension from the circle $\mathbb{T}$ to the unit sphere $\mathbb{S}^{n-1}\subset \mathbb{R}^n$.

Before stating the main results, we need to introduce some notation.

For notational convenience, given a point 
$x=(x_1,\dots,x_n)\in\mathbb{R}^n$, we write
\[
x' := (x_1,\dots,x_{n-1}) \in \mathbb{R}^{\,n-1},
\qquad
x'' := (x_3,\dots,x_{n-1}) \in \mathbb{R}^{\,n-3}.
\]
Let $\{E_i^\mathbb{T}\}_{i=1}^r$ be the measurable partition of $\mathbb{T}$ provided by Theorem~\ref{thm:KKT-circle}.  

We also introduce the rotation
\begin{align}\label{definition-rotation}
R_\theta(x_1,x_2,x'',x_n)
 :=\bigl(\,
        x_1\cos(2\pi\theta) - x_2\sin(2\pi\theta),\ 
        x_1\sin(2\pi\theta) + x_2\cos(2\pi\theta),\ 
        x'',\ x_n
      \,\bigr),
\end{align}
that is, rotation by angle $2\pi\theta$ in the $(x_1,x_2)$–plane.

We now state our first main result for spheres.


\begin{theorem}[Spheres]\label{thm:sphere-Raimi}
Let $r, t\ge 2$ be numbers in $\mathbb{N}$.
Let $n\ge 3$,
\[
\mathbb{S}^{n-1} := \{x\in\mathbb{R}^n : |x|=1\}
\]
be the unit sphere equipped with the normalized surface measure $\sigma_{n-1}$.
Then there exists a measurable partition $\{E_i^{\mathbb{S}^{n-1}}\}_{i=1}^r$ of $\mathbb{S}^{n-1}$ such that:
For every finite measurable cover
\[
\mathbb{S}^{n-1} \subset F_1\cup\cdots\cup F_t,
\]
there exist an index $m_0\in\{1,\dots,t\}$ and $\theta_0\in \mathbb{T}$ such that
\[
\sigma_{n-1}\big(R_{\theta_0}(F_{m_0})
    \cap E_i^{\mathbb{S}^{n-1}}\big)>0
\qquad\text{for all }1\le i\le r.
\]
where $R_\theta$ denotes the rotation defined in \eqref{definition-rotation} as restricted to $\mathbb{S}^{n-1}$.
\end{theorem}


The next theorem establishes the result for rotational power surfaces, with cones and paraboloids arising as special cases.
\begin{theorem}[Rotational power surfaces]\label{thm:power-surface-Raimi}
Let $r, t\ge 2$ be numbers in $\mathbb{N}$. 
Let $n\ge 3$ and $k>0$. Define
\[
\mathcal{S}_{k,R}
 := \big\{(x',x_n)\in\mathbb{R}^{n-1}\times\mathbb{R}
        : x_n = |x'|^{\,k},\ 0<|x'|\le R\big\}.
\]
Let
$\sigma_{k,R}$ be the normalized surface measure on $\mathcal{S}_{k,R}$.
Then there exists a measurable partition
$\{E_i^{\mathcal{S}_{k,R}}\}_{i=1}^r$ of $\mathcal{S}_{k,R}$ with the following property:
For every finite measurable cover
\[
\mathcal{S}_{k,R} \subset F_1\cup\cdots\cup F_t
\]
there exist an index $m_0\in\{1,\dots,t\}$ and $\theta_0\in \mathbb{T}$ such that
\[
\sigma_{k,R}\big(R_{\theta_0}(F_{m_0})
       \cap E_i^{\mathcal{S}_{k,R}}\big)>0
\qquad\text{for all }1\le i\le r.
\]
In particular, $k=1$ yields the cone, and $k=2$ yields the paraboloid.
\end{theorem}


We conclude with a theorem addressing the cylindrical case.
\begin{theorem}[Cylindrical surface]\label{thm:cylinder-Raimi}
Let $r, t\ge 2$ be numbers in $\mathbb{N}$. 
Let $n\ge 3$, $R>0$, and $\Omega\subset\mathbb{R}^{n-2}$ be a bounded Borel set. Define the cylindrical surface
\[
\mathcal{C}_{R,\Omega}
:= \big\{(x_1,x_2, x'', x_n)\in\mathbb{R}^2\times\mathbb{R}^{n-2}
 : x_1^2+x_2^2=R^2,\ (x'', x_n)\in\Omega\big\}.
\]
Let $\mu_{\text{\tiny $R,\Omega$}}$ be the normalized surface measure on $\mathcal{C}_{R,\Omega}$.
Then there exists a measurable partition $\{E_i^{\mathcal{C}_{R,\Omega}}\}_{i=1}^r$ of $\mathcal{C}_{R,\Omega}$ with the following property:

For every finite measurable cover
\[
\mathcal{C}_{R,\Omega} \subset F_1\cup\cdots\cup F_t
\]
there exist an index $m_0\in\{1,\dots,t\}$ and $\theta_0\in \mathbb{T}$ such that
\[
\mu_{\text{\tiny $R,\Omega$}}\big(R_{\theta_0}(F_{m_0})\cap E_i^{\mathcal{C}_{R,\Omega}}\big)>0
\qquad\text{for all }1\le i\le r,
\]
where $R_\theta$ denotes the rotation defined in \eqref{definition-rotation} as restricted to $\mathcal{C}_{R,\Omega}$.
\end{theorem}

For readers with a combinatorial background, we stress that the analytic framework employed here is conceptually minimal. The only genuinely measure–theoretic ingredient is the standard disintegration theorem for probability measures (Rokhlin’s theorem). Once the surface measure is decomposed along circle orbits, rotational invariance forces the conditional measures on almost every fibre to coincide with Haar measure on the circle. After this structural step, the remainder of the argument is elementary: a Fubini-type slicing argument, a pigeonhole principle on the base circle, and an application of the measurable Raimi partition established in \S\ref{section-cylinder}.

Therefore, while the problem is formulated geometrically, its essential mechanism is combinatorial in nature: unavoidable partitions of the circle are transferred along circle fibres, yielding unavoidable partitions of the manifold.

\subsection{Intuition and Proof Strategy}

The extension of Raimi's theorem beyond the circle is guided by a simple structural idea: if a geometric surface carries a measure-preserving circle action whose orbits are one-dimensional, then the space behaves, up to a null set, like a measurable product
\[
X \simeq \mathbb{T} \times Y.
\]
This hidden circle factor allows the one-dimensional Raimi phenomenon to be transported to higher-dimensional settings.

\paragraph{Why rotational symmetry is the natural setting.}
Raimi’s theorem concerns translations on $\mathbb{T}$.
To extend this phenomenon beyond groups, it is therefore natural to consider spaces that possess a distinguished circular direction along which translations act intrinsically.
The sphere $\mathbb{S}^{n-1}$, rotational power surfaces,
and circular cylinders all share this feature:
rotation in the $(x_1,x_2)$–plane preserves the surface and preserves its natural surface measure.
Moreover, away from a lower-dimensional singular set,
each point lies on a circular orbit of this action.
This makes such surfaces the natural geometric setting
for lifting the circle-based combinatorial phenomenon.

\paragraph{A concrete two-dimensional example.}
To illustrate the mechanism, consider the torus
\[
\mathbb{T}^2 = \mathbb{T} \times \mathbb{T}
\]
with rotation acting only on the first coordinate:
\[
R_\alpha(\theta_1,\theta_2) = (\theta_1+\alpha,\theta_2).
\]
Here the product structure is explicit:
$\theta_1$ parametrizes the circle fibres,
while $\theta_2$ plays the role of the base variable.
A measurable partition $\{E_i^{\mathbb{T}}\}$ of the first circle immediately induces a partition of the torus via
\[
E_i^{\mathbb{T}^2} = E_i^{\mathbb{T}} \times \mathbb{T}.
\]
The Raimi property on $\mathbb{T}$ then lifts directly to $\mathbb{T}^2$ by integrating over the second coordinate.
The geometric surfaces treated in this paper behave measurably in exactly this way, even though the product structure is no longer globally smooth.

\paragraph{How the lifting mechanism works.}
Suppose $X$ admits a measurable parametrization
\[
\Phi : \mathbb{T} \times Y \longrightarrow X \setminus N
\]
intertwining translation with geometric rotation:
\[
\Phi(\theta+\alpha,y)=R_\alpha(\Phi(\theta,y)).
\]
Given a finite measurable cover $\{F_m\}_{m=1}^t$ of $X$,
we pull it back via $\Phi$ to obtain a cover of $\mathbb{T}\times Y$.
For each angle $\theta$, we examine the fibre slice
\[
A_m(\theta)=\{y\in Y:(\theta,y)\in\Phi^{-1}(F_m)\}.
\]
Since the slices cover $Y$, a pigeonhole argument implies that for every $\theta$, some index $m$ satisfies $\nu(A_m(\theta))\ge \frac{1}{t}$.
This produces a measurable partition of the base circle.
Applying the Raimi theorem on $\mathbb{T}$ then yields a rotation forcing one cover element to intersect every partition component of the circle in positive measure.
Integrating back over $Y$ yields positive surface-measure intersections in $X$.

\paragraph{Why measure theory is essential.}
The product structure of $X$ is not smooth in the usual sense. For example, on the sphere the circle orbits degenerate along a codimension-two set, and rotational power surfaces have a singular axis.
Thus the decomposition
\[
d\mu_X(x)=d\mu_{\mathbb{T}}(\theta)\,d\nu_Y(y)
\]
exists only in a measurable sense. Rokhlin’s disintegration theorem ensures that the invariant surface measure splits into conditional measures along circle fibres.
Rotational invariance then forces these conditional measures to coincide with Haar measure on $\mathbb{T}$.
This is the only genuinely analytic ingredient;
once established, the argument is entirely combinatorial.

\paragraph{Outline of the geometric proofs.}
The proofs of Theorems~\ref{thm:sphere-Raimi}–\ref{thm:cylinder-Raimi} follow a unified scheme as follow:
\begin{itemize}
\item Remove a null singular set $N$ where the circle action degenerates.
\item Construct a measurable bijection
      \[
      \Phi:\mathbb{T}\times Y\to X\setminus N.
      \]
\item Verify equivariance under rotations.
\item Use disintegration to show that the surface measure factorizes.
\item Apply the measurable Raimi partition on $\mathbb{T}$.
\end{itemize}

Thus the proofs reduce to verifying the structural hypotheses of the general circle-bundle theorem introduced in Section~\ref{section-main-result}.
Once this framework is in place, the partition constructed in the previous work of Kang, Koh, and the author on the circle $\mathbb{T}$ automatically induces the desired Raimi-type partitions on $\mathbb{S}^{n-1}$,
on rotational power surfaces,
and on circular cylindrical surfaces.

\medskip

Our approach naturally leads to the following open question: 
Let $M$ be a compact hyperbolic surface with its normalized area measure. Does $M$ admit a measurable Raimi--type partition (in the sense of this paper), even though it has no circle action and therefore lies outside our circle--bundle framework?

{\bf The paper is organized as follows.}
In Section~\ref{section-main-result}, we establish the general circle-bundle theorem, which serves as the main structural tool of the paper. Sections~\ref{section-sphere}, \ref{section-power-surface}, and \ref{section-cylinder} are devoted to the proofs of the three principal applications: the sphere, the rotational power surfaces, and the circular cylindrical surfaces, respectively.  Each result follows by verifying the hypotheses of the general theorem and applying the measurable Raimi partition on the base circle.


\section{Raimi partitions on circle bundles}\label{section-main-result}

We now present the following general theorem, of which the spherical, cylindrical, and rotational power surface cases are specific instances.

The key idea is simple: if a surface $X$ looks like
circle $\mathbb{T} \times Y$ base locally, and rotations act by rotating the circle factor, then we can lift
Raimi partitions from $\mathbb{T} \times Y$ to $X$. The following theorem makes this precise.
\begin{theorem}\label{thm:general-circle-bundle}
Let $(X,\mu)$ be a probability space equipped with a measurable, measure-preserving action
\[
\{R_\theta\}_{\theta\in \mathbb{T}},\qquad \mathbb{T}:=\mathbb{R}/\mathbb{Z},
\]
of the circle group $\mathbb{T}$. Assume there exist
\begin{itemize}
    \item a probability space $(Y,\nu)$,
    \item a measurable set $N\subset X$ with $\mu(N)=0$,
    \item a measurable bijection
    \[
        \Phi : \mathbb{T} \times Y \longrightarrow X\setminus N,
    \]
\end{itemize}
satisfying:
\begin{enumerate}
    \item[(i)] (\emph{Equivariance}) For all $\theta,\alpha\in \mathbb{T}$ and $y\in Y$,
    \[
        \Phi(\theta+\alpha,y) = R_\alpha(\Phi(\theta,y)).
    \]
    \item[(ii)] (\emph{Product disintegration}) For every bounded measurable function $f:X\to\mathbb{R}$,
    \[
        \int_X f(x)\,d\mu(x)
        = \int_\mathbb{T} \int_Y f(\Phi(\theta,y))\,d\nu(y)\,d\mu_1(\theta),
    \]
    where $\mu_1$ is the normalized Lebesgue measure on $\mathbb{T}$.
\end{enumerate}
Let $\{E_i^\mathbb{T}\}_{i=1}^r$ be a measurable partition of $\mathbb{T}$ with the Raimi property from Theorem~\ref{thm:KKT-circle}. Define 
\begin{equation*}
    \begin{cases}
        E_1^X := \Phi(E_1^\mathbb{T}\times Y) \cup N,\\
        E_i^X:=\Phi(E_i^\mathbb{T}\times Y), ~~~~~~~~2 \leq i\leq r.
    \end{cases}
\end{equation*}
Then $\{E_i^X\}_{i=1}^r$ is a measurable partition of $X$ with the following property:

For every finite measurable cover $\{F_m\}_{m=1}^t$ of $X$, there exist an index $m_0\in\{1,\dots,t\}$ and a rotation $R_{\theta_0}$ with $\theta_0\in \mathbb{T}$ such that
\[
    \mu\big(R_{\theta_0}(F_{m_0})\cap E_i^X\big) > 0
    \qquad\text{for all } 1\le i\le r.
\]
\end{theorem}

\begin{figure}[h]
\centering
\begin{tikzpicture}[scale=1, >=Stealth]

\node at (-6,5.5) {$\mathbb{T}\times Y$};

\draw[very thick, fill=orange!30, draw=orange!80!black] (-6,4) circle (0.7);
\node at (-6,4) {$\color{orange}\mathbb{T}$};

\draw[->, thick] (-5.2,4.6) arc (40:120:0.8);
\node at (-4.9,4.2) {$+\alpha$};

\fill[red] (-6,2.5) circle (2pt);
\node[right] at (-5.8,2.5) {$\color{red}(\theta,y)$};

\draw[dashed] (-6,3.25) -- (-6,-2);
\draw[dashed, gray!60] (-6,3.25) -- (-6,3.8);
\node[left] at (-6.7,2.2) {\scriptsize fiber};
\draw[dashed, gray!70] (-6.6,3.7) -- (-6.6,0.5);

\draw[thick, fill=green!5, draw=green!80!black] (-6,-2) ellipse (1.6 and 0.6);
\node at (-6,-2) {$Y$};
\fill[green!50!black] (-6.7,-2) circle (2pt);
\node[left] at (-6.9,-2.0) {\scriptsize $y$};

\draw[->, very thick, blue] (-4,1) -- (0,1);
\node[above] at (-2,1.1) {\color{blue}$\Phi$};
\node[below] at (-1.2,0.8) {\scriptsize bijection to $X\setminus N$};
\node at (4,4.5) {$X$ (\textit{manifold})};  

\coordinate (A) at (2.4,-1.2);    
\coordinate (B) at (5.6,-1.2);    
\coordinate (C) at (1.6,3.3);    
\coordinate (D) at (6.4,3.3);   

\draw[ultra thick, blue!80!black] (A) to[out=-80, in=-100] (B);
\draw[ultra thick, blue!80!black] (C) to[out=-20, in=-200] (D);
\draw[blue!80!black, ultra thick] (1.6,3.3) -- (2.4,-1.2);    
\draw[blue!80!black, ultra thick] (6.4,3.3) -- (5.6,-1.2);    

\fill[blue!80!yellow!10] 
    (A) to[out=-80, in=-100] (B) 
    -- (D) 
    to[out=-200, in=-20] (C) 
    -- cycle;

\draw[dashed, red!80!black, thick] (4,2) ellipse (1.4 and 0.6); 

\fill[purple] (4.5,1.55) circle (2.4pt);         
\node[right] at (5.4,2) {\color{red}$\Phi(\theta,y)$};   
\node[below] at (4.5,1.2) {\color{purple}$R_\alpha(\Phi(\theta,y))$}; 
\draw[very thin, fill=red, draw=black] (5.1,2) circle (3.0pt);  

\draw[->, thick] (5.3,2.6) arc (30:120:0.6);      
\node[right] at (5.3,2.6) {\color{blue}$R\_\alpha$};

\fill[red!60] (4,-0.5) ellipse (0.3 and 0.15);     
\node at (4.7,-0.5) {\color{red}$N$};             
\node at (4.7,-1.1) {\color{red}(\scriptsize $\mu(N)=0$)}; 

\node[draw=yellow!80!black, fill=yellow!20, rounded corners] at (4,-3)  
{\color{red}\scriptsize Circle orbit: $\{R_\alpha(x):\alpha\in\mathbb{T}\}$};

\node[draw, rounded corners, fill=gray!10, minimum width=16cm]
at (-1,-6)
{
\begin{minipage}{15cm}
\vspace{0.3cm}
\begin{center}
    \textbf{Key Properties}
\end{center}

\textbf{\color{blue}(i) Equivariance:} $\quad \Phi(\theta+\alpha,y)=R_\alpha(\Phi(\theta,y))$

\medskip
\textbf{\color{blue}(ii) Product disintegration:}
\[
\int_X f(x)\,d\mu(x)
=\int_{\mathbb{T}}\int_Y f(\Phi(\theta,y))\,d\nu(y)\,d\mu_{\mathbb{T}}(\theta)
\]
\vspace{0.2cm}
\end{minipage}
};

\end{tikzpicture}
 \caption{Circle bundle structure for Theorem \ref{thm:general-circle-bundle}}
\end{figure}

\begin{proof}

Since $\Phi : \mathbb{T}\times Y \to X\setminus N$ is a measurable bijection and $\{E_i^\mathbb{T}\}_{i=1}^r$ is a measurable partition of $\mathbb{T}$, it follows that
\[
\{\Phi(E_i^\mathbb{T}\times Y)\}_{i=1}^r
\]
is a measurable partition of $X\setminus N$. 

By condition (i) of the theorem, the action $\{R_\alpha\}$ and the
parametrization $\Phi$ are related by
\[
    R_\alpha(\Phi(\theta,y)) = \Phi(\theta+\alpha,y)
    \qquad\text{for all }\theta,\alpha\in \mathbb{T},\ y\in Y.
\]
Since $\Phi$ is bijective from $\mathbb{T}\times Y$ onto $X\setminus N$, this identity determines the action on $X\setminus N$ uniquely; we extend it to all of $X$ by letting $R_\alpha$ act as the identity on $N$.

Adding $N$ to $E_1^X$ therefore yields
\[
X = (X\setminus N)\cup N = \bigcup_{i=1}^r E_i^X,
\]
with the $E_i^X$ pairwise disjoint.  
Thus, $\{E_i^X\}_{i=1}^r$ is a measurable partition of $X$.

Let $\{F_m\}_{m=1}^t$ be a finite measurable cover of $X$.
Since $\mu(N)=0$, all measure-theoretic computations may be carried out on $X\setminus N$ without loss of generality. In particular, we may restrict attention to the cover
\[
X\setminus N \subset \bigcup_{m=1}^t F_m.
\]
As $\Phi$ is a measurable bijection from $\mathbb{T}\times Y$ onto $X\setminus N$, we may pull back the cover by defining
\[
F_m^{\mathrm e}
   := \Phi^{-1}\bigl(F_m\cap (X\setminus N)\bigr)
   \subset \mathbb{T}\times Y, \qquad 1\le m\le t.
\]
Then $\{F_m^{\mathrm e}\}_{m=1}^t$ is a measurable cover of $\mathbb{T}\times Y$.

Writing points of $\mathbb{T}\times Y$ as $(\theta,y)$, for each $\theta\in \mathbb{T}$, we define the slices
\[
A_m(\theta) := \{\, y\in Y : (\theta,y)\in F_m^{\mathrm e}\,\}, \qquad 1\le m\le t.
\]
Since the sets $F_m^{\mathrm e}$ cover $\mathbb{T}\times Y$, it follows that
\[
\bigcup_{m=1}^t A_m(\theta) = Y
\qquad\text{for every } \theta\in \mathbb{T}.
\]
For each fixed $\theta\in \mathbb{T}$, we have
\[
    \sum_{m=1}^t \nu(A_m(\theta))
    \ge \nu\Big(\bigcup_{m=1}^t A_m(\theta)\Big)
    = \nu(Y)=1.
\]
By the pigeonhole principle, for each $\theta\in \mathbb{T}$, there exists at least one index $m\in\{1,\dots,t\}$ such that
\[
    \nu(A_m(\theta)) \ge \frac1t.
\]
Define an index function $m(\theta)$ by
\begin{align*}
    m: \mathbb{T} \to & \{1,\dots,t\},\\
    \theta \mapsto & m(\theta) := \min\Big\{1\le m\le t : \nu(A_m(\theta))\ge \frac1t\Big\}.
\end{align*}
We now verify that $m$ is measurable.
Since each $F_m^{\mathrm e}$ is measurable in $\mathbb{T}\times Y$, its indicator $\mathbf{1}_{F_m^{\mathrm e}}(\theta,y)$ is measurable. By Fubini's theorem, the function
\[
    \theta \mapsto \int_Y \mathbf{1}_{F_m^{\mathrm e}}(\theta,y)\,d\nu(y)
\]
is measurable, and this integral equals
\[
    \int_Y \mathbf{1}_{F_m^{\mathrm e}}(\theta,y)\,d\nu(y)
    = \nu\Big(\{y:(\theta,y)\in F_m^{\mathrm e}\}\Big)
    = \nu(A_m(\theta)).
\]
Hence for each $m$, the map $\theta\mapsto\nu(A_m(\theta))$ is measurable. Therefore, the sets
\[
    \bigg\{\theta : \nu(A_m(\theta))\ge \frac{1}{t} \bigg\}
\]
are measurable. Taking the minimum over finitely many measurable sets shows that $\theta\mapsto m(\theta)$ is measurable.

Now define a measurable partition of $\mathbb{T}$ by
\[
    C_m := \{\theta\in \mathbb{T}: m(\theta)=m\},\qquad 1\le m\le t.
\]
By construction, for every $\theta\in C_m$ we have
\begin{equation}\label{eq:Am-lower-bound}
    \nu(A_m(\theta))\ge \frac1t.
\end{equation}
By applying Theorem~ \ref{thm:KKT-circle} to the partition $\{C_m\}_{m=1}^t$ of $\mathbb{T}$, there exist an index $m_0\in\{1,\dots,t\}$ and a rotation
\[
    R_{\theta_0}: \mathbb{T} \to \mathbb{T},\qquad R_{\theta_0}(\theta)=\theta+\theta_0,
\]
such that
\begin{equation}\label{eq:KKT-circle-hit}
    \mu_1\big(R_{\theta_0}(C_{m_0})\cap E_i^\mathbb{T}\big)>0,
    \qquad\, \forall\, 1\le i\le r.
\end{equation}

To complete the proof of the theorem, it remains to use \eqref{eq:KKT-circle-hit} to show that
\begin{align}\label{ineq:intersection-rotation-E_i^X}
        \mu\bigl(R_{\theta_0}(F_{m_0}) \cap E_i^X\bigr) > 0
    \qquad \text{for all } 1 \le i \le r.
\end{align}
Let 
\[
    R_{\theta_0}(\theta,y):=(\theta+\theta_0,y)
\]
be the induced rotation on $\mathbb{T}\times Y$. By equivariance, we have 
\[
    R_{\theta_0}(F_{m_0})\cap (X\setminus N)
    = R_{\theta_0}\big(\Phi(F_{m_0}^{\mathrm e})\big)
    = \Phi\big(R_{\theta_0}(F_{m_0}^{\mathrm e})\big),
\]
which, together with the definition of $E_i^X$, gives
\[
    R_{\theta_0}(F_{m_0})\cap E_i^X
    = \Phi\Big(R_{\theta_0}(F_{m_0}^{\mathrm e})\cap (E_i^\mathbb{T}\times Y)\Big)
        \cup \Big(R_{\theta_0}(F_{m_0})\cap N\cap E_i^X\Big).
\]
By the assumption on $N$, the definition of $F_{m_0}^{\mathrm e}$, and the rotation on $\mathbb{T}\times Y$, we have
\begin{align*}
    \mu\big(R_{\theta_0}(F_{m_0})\cap E_i^X\big)
    =& (\mu_1\times\nu)\big(R_{\theta_0}(F_{m_0}^{\mathrm e})\cap (E_i^\mathbb{T}\times Y)\big) \\
    = &\int_{\mathbb{T}}\int_Y 
        \mathbf{1}_{R_{\theta_0}(F_{m_0}^{\mathrm e})}(\theta,y)
        \mathbf{1}_{E_i^\mathbb{T}}(\theta)\,d\nu(y)\,d\mu_1(\theta) \\
    =& \int_{\mathbb{T}}\int_Y 
        \mathbf{1}_{A_{m_0}(\theta-\theta_0)}(y)\mathbf{1}_{E_i^\mathbb{T}}(\theta)
        \,d\nu(y)\,d\mu_1(\theta) \\
    =& \int_{\mathbb{T}} \nu\big(A_{m_0}(\theta-\theta_0)\big)\,\mathbf{1}_{E_i^\mathbb{T}}(\theta)\,d\mu_1(\theta).
\end{align*}
By restricting the outer integral to those $\theta$ satisfying 
$\theta-\theta_0\in C_{m_0}$, i.e.\ $\theta\in R_{\theta_0}(C_{m_0})$, and using \eqref{eq:Am-lower-bound} together with \eqref{eq:KKT-circle-hit}, we obtain
\begin{align*}
    \mu\big(R_{\theta_0}(F_{m_0})\cap E_i^X\big)
    \geq & \int_{R_{\theta_0}(C_{m_0})\cap E_i^\mathbb{T}} \nu\big(A_{m_0}(\theta-\theta_0)\big)\,d\mu_1(\theta) \\
    &\ge \frac1t\,\mu_1\big(R_{\theta_0}(C_{m_0})\cap E_i^\mathbb{T}\big)>0
\end{align*}
This is precisely \eqref{ineq:intersection-rotation-E_i^X}. This completes the proof.
\end{proof}

\begin{lemma}[Invariant conditional measures on circle fibres]\label{lem:Haar-fibres}
Let $(X,\mu)$ be a probability space equipped with a measurable, measure-preserving action of the circle group $\mathbb{T}$.
Suppose that $\pi:X\to Y$ is a measurable factor map whose fibres $\pi^{-1}(\{y\})$ are $\mathbb{T}$--orbits, and let $\{\mu_y\}_{y\in Y}$ be a Rokhlin disintegration of $\mu$ over $Y$.
Then for $\nu$--almost every $y\in Y$, the conditional measure $\mu_y$ is invariant under translations on the fibre. Consequently,
\[
\mu_y = \mu_1 \qquad \text{for $\nu$--a.e.\ } y\in Y,
\]
where $\mu_1$ denotes the normalized Haar measure on $\mathbb{T}$.
\end{lemma}
\begin{proof}
Since $\mu$ is invariant under the $\mathbb{T}$--action, the Rokhlin
conditional measures $\mu_y$ must be invariant under translations
along the fibres for $\nu$--almost every $y$.
By uniqueness of the translation-invariant probability measure on
$\mathbb{T}$, it follows that $\mu_y=\mu_1$.
\end{proof}


\section{Proof of Theorem \ref{thm:sphere-Raimi} (Spheres)}\label{section-sphere}

We apply Theorem~\ref{thm:general-circle-bundle} with
$X=\mathbb{S}^{n-1}$ and $\mu=\sigma_{n-1}$. The action 
$\{R_\theta\}_{\theta\in \mathbb{T}}$ is the rotation in the $(x_1,x_2)$–plane, which is measurable and $\sigma_{n-1}$–preserving.


\medskip\noindent
\emph{We now specify $Y$, $N$, and $\Phi$ for Theorem \ref{thm:general-circle-bundle}.}
Define
\[
N := \{x\in\mathbb{S}^{n-1} : x_1=x_2=0\},
\]
which is a spherical subset of codimension $2$ and hence satisfies
$\sigma_{n-1}(N)=0$. Let $Y := \mathbf{B}_{n-2}$ denote the open unit ball in $\mathbb{R}^{\,n-2}$,
\[
Y = \{v \in \mathbb{R}^{\,n-2} : |v| < 1\},
\]
and set
$r(v):=\sqrt{1-|v|^2}$ for $v\in Y$.  
We define the parametrization
\begin{align}\label{definition:Phi-sphere}
    \Phi: \mathbb{T} \times Y \to  \, \mathbb{S}^{n-1}\setminus N, ~~~~ \Phi(\theta,v)
 := \big(r(v)\cos(2\pi\theta),\, r(v)\sin(2\pi\theta),\, v\big).
\end{align}
A direct computation shows $|\Phi(\theta,v)|=1$, and $\Phi$ is clearly measurable.

\emph{$\Phi$ is a bijection.}
Observe that every point $x\in\mathbb{S}^{n-1}\setminus N$ admits a unique polar representation in the $(x_1,x_2)$–plane. Writing
\[
x=\big(r\cos(2\pi\theta),\, r\sin(2\pi\theta),\, v\big),
\qquad r=\sqrt{x_1^2+x_2^2}>0,\ v=(x'',x_n)\in Y,
\]
we see that $x=\Phi(\theta,v)$, so $\Phi$ is surjective.
Conversely, if $\Phi(\theta,v)=\Phi(\theta',v')$ then $v=v'$ from the last $n-2$ coordinates, and since $r(v)>0$,
\[
(\cos 2\pi\theta,\sin 2\pi\theta)
 =(\cos 2\pi\theta',\sin 2\pi\theta'),
\]
which implies $\theta=\theta'$ in $\mathbb{T}$. Thus, $\Phi$ is injective, and therefore a bijection from $\mathbb{T} \times Y$ onto $\mathbb{S}^{\,n-1}\setminus N$.

\medskip\noindent
\emph{Verification of equivariance (assumption (i)).}
By construction,
\begin{align*}
R_\alpha(\Phi(\theta,v))
&= R_\alpha\big(r(v)\cos 2\pi\theta,\ r(v)\sin 2\pi\theta,\ v\big)\\
&= \big(r(v)\cos 2\pi(\theta+\alpha),\ r(v)\sin 2\pi(\theta+\alpha),\ v\big)\\
&= \Phi(\theta+\alpha,v),
\end{align*}
for all $\theta,\alpha\in \mathbb{T}$ and $v\in Y$. This is exactly condition~(i) in Theorem~\ref{thm:general-circle-bundle}.


\medskip\noindent
\emph{Product disintegration (assumption (ii)).}
Consider the projection
\[
\pi:\mathbb{S}^{n-1}\setminus N\to Y,\qquad
\pi(x_1,x_2, x'', x_n) := (x'', x_n),
\]
so that $\pi(\Phi(\theta,v))=v$. Define $\nu$ as the pushforward of
$\sigma_{n-1}$ under $\pi$:
\begin{align}\label{definition:measure-nu-sphere}
    \nu(A) := \sigma_{n-1}\big(\pi^{-1}(A)\big),
\qquad \text{ for every Borel set } ~~ A\subset Y.
\end{align}
Since the projection $\pi$ collapses each rotational orbit to a single point in $Y$, the conditional measures along the fibres must reflect the rotational invariance  of $\sigma_{n-1}$. By Rokhlin’s disintegration theorem, there exists a family of conditional probability measures $\{\mu_v\}_{v\in Y}$ on $\mathbb{T}$ such that, for every bounded measurable $g: \mathbb{T}\times Y\to\mathbb{R}$,
\begin{equation}\label{eq:sphere-disintegration}
\int_{\mathbb{S}^{n-1}} g(\theta(x),\pi(x))\,d\sigma_{n-1}(x)
= \int_Y\int_\mathbb{T} g(\theta,v)\,d\mu_v(\theta)\,d\nu(v),
\end{equation}
where $\theta(x)$ is any measurable choice of $\theta$ with
$x=\Phi(\theta,\pi(x))$.

For each fixed $v\in Y$, the fibre $\pi^{-1}(\{v\})$ is the circle
\[
\{(x_1,x_2,v)\in\mathbb{R}^n : x_1^2+x_2^2=r(v)^2\}
\]
parametrized by $\theta\mapsto\Phi(\theta,v)$. The action $\{R_\alpha\}$ acts transitively on each fibre by $\theta\mapsto\theta+\alpha$ and preserves $\sigma_{n-1}$.

 By Lemma~\ref{lem:Haar-fibres}, the conditional measures on the circle fibres coincide with the normalized Haar measure $\mu_1$. Therefore, for $\nu$-almost every $v$, the conditional measure $\mu_v$ must be invariant under all such rotations.
The only probability measure on $\mathbb{T}$ with this invariance property is the normalized Lebesgue measure $\mu_1$. Hence
\[
\mu_v = \mu_1\qquad\text{for $\nu$-a.e. }v\in Y.
\]
By choosing \(g(\theta,v):=f(\Phi(\theta,v))\) in
\eqref{eq:sphere-disintegration}, and using that \(\Phi\) is a bijection from \(\mathbb{T}\times Y\) onto
\(\mathbb{S}^{n-1}\setminus N\) with \(N\) of measure zero, we obtain, for every bounded measurable $f:\mathbb{S}^{n-1}\to\mathbb{R}$,
\begin{align*}
\int_{\mathbb{S}^{n-1}} f(x)\,d\sigma_{n-1}(x)
= \int_Y\int_\mathbb{T} f(\Phi(\theta,v))\,d\mu_1(\theta)\,d\nu(v)= \int_\mathbb{T} \int_Y f(\Phi(\theta,v))\,d\nu(v)\,d\mu_1(\theta).
\end{align*}
This is precisely condition~(ii) in Theorem~\ref{thm:general-circle-bundle}.

\medskip\noindent
With the structural assumptions of Theorem~\ref{thm:general-circle-bundle} now confirmed for 
\[
X=\mathbb{S}^{n-1},\quad 
\mu=\sigma_{n-1},\quad
Y=\mathbf{B}_{n-2},\quad
\nu \text{ defined in \eqref{definition:measure-nu-sphere}},\quad
N=\{x\in\mathbb{S}^{n-1}: x_1=x_2=0\},
\]
and with $\Phi$ given in \eqref{definition:Phi-sphere}, the theorem furnishes a measurable partition
\[
E_i^{\mathbb{S}^{n-1}}
   =\Phi(E_i^\mathbb{T}\times Y),\qquad 1\le i\le r,
\]
where \(\{E_i^\mathbb{T} \}_{i=1}^r\) arises from Theorem~\ref{thm:KKT-circle}. 
The desired Raimi property follows immediately from Theorem \ref{thm:general-circle-bundle}, completing the proof


\section{Proof of Theorem \ref{thm:power-surface-Raimi} (Rotational power surfaces)}\label{section-power-surface}

We apply Theorem 2.1 with $X = \mathcal{S}_{k,R}$ and $\mu = \sigma_{k,R}$. By rotational symmetry, the maps $\{R_\theta\}_{\theta \in \mathbb{T}}$ act measurably on $\mathcal{S}_{k,R}$ and preserve $\sigma_{k,R}$, the uniform probability measure on the surface.

\emph{Choice of $Y, N$, and $\Phi_k$.} To handle the singular axis, we define
\[ N := \{x \in \mathcal{S}_{k,R} : x_1 = x_2 = 0\}, \]
which is a subset of codimension 2 in the $(n-1)$-dimensional surface, hence $\sigma_{k,R}(N) = 0$. Moreover, $N$ is invariant under all $R_\theta$.

Recall $x'= (x_1, x_2, x'')$ with $x'' \in \mathbb{R}^{n-3}$, and let $\rho = \sqrt{x_1^2 + x_2^2}$ be the radial coordinate in the rotation plane. Then $|x'|^2 = \rho^2 + |x''|^2$, and the condition $0 < |x'| \le R$ is equivalent to $\rho^2 + |x''|^2 \le R^2$. Thus, we take
\begin{equation}\label{definition-Y}
Y := \left\{(\rho, x'') \in (0, \infty) \times \mathbb{R}^{n-3} : \rho^2 + |x''|^2 \le R^2\right\}.
\end{equation}

Define the parametrization $\Phi_k : \mathbb{T} \times Y \rightarrow \mathbb{R}^n$ by
\begin{equation}\label{definition:Phi_k-paraboloid}
\Phi_k(\theta, \rho, x'') := \left(\rho \cos(2\pi\theta), \rho \sin(2\pi\theta), x'', (\rho^2 + |x''|^2)^{k/2}\right).
\end{equation}
A direct computation shows that $x_n = (\rho^2 + |x''|^2)^{k/2} = |x'|^k$ and $|x'|^2 = \rho^2 + |x''|^2 \le R^2$; hence $\Phi_k(\theta, \rho, x'') \in \mathcal{S}_{k,R}$, so $\Phi_k$ is well defined.

\emph{$\Phi_k$ is a measurable bijection} (i.e., a one-to-one correspondence preserving measurable sets).
The restriction $\rho > 0$ excludes the singular axis $N$ and guarantees that every point in $\mathcal{S}_{k,R} \setminus N$ admits a unique polar representation in the $(x_1, x_2)$-plane. Indeed, every point $x = (x', x_n) \in \mathcal{S}_{k,R} \setminus N$ satisfies $x_n = |x'|^k$ and can be written uniquely as
\[ x' = (\rho \cos(2\pi\theta), \rho \sin(2\pi\theta), x''), \quad \rho = \sqrt{x_1^2 + x_2^2} > 0, \]
with $(\rho, x'') \in Y$. Hence $x = \Phi_k(\theta, \rho, x'')$, so $\Phi_k$ is surjective.

Conversely, if $\Phi_k(\theta, \rho, x'') = \Phi_k(\bar\theta, \bar\rho, \bar{x}'')$, then comparison of the last $n-2$ coordinates yields $(\rho, x'') = (\bar\rho, \bar{x}'')$. Since $\rho > 0$, uniqueness of polar coordinates implies $\theta = \bar\theta$ in $\mathbb{T}$. Thus $\Phi_k$ is injective. Therefore, $\Phi_k$ is a measurable bijection from $\mathbb{T} \times Y$ onto $\mathcal{S}_{k,R} \setminus N$.

\emph{Verification of equivariance (assumption (i)).} Since $R_\alpha$ only rotates the $(x_1, x_2)$-coordinates, we obtain
\begin{align*}
R_\alpha(\Phi_k(\theta, \rho, x'')) &= R_\alpha\left(\rho \cos(2\pi\theta), \rho \sin(2\pi\theta), x'', (\rho^2 + |x''|^2)^{k/2}\right) \\
&= \left(\rho \cos(2\pi(\theta+\alpha)), \rho \sin(2\pi(\theta+\alpha)), x'', (\rho^2 + |x''|^2)^{k/2}\right) \\
&= \Phi_k(\theta+\alpha, \rho, x'').
\end{align*}
Thus, $\Phi_k$ satisfies the equivariance condition (i) in Theorem 2.1.

\emph{Product disintegration (assumption (ii)).} Consider the projection $\pi : \mathcal{S}_{k,R} \setminus N \rightarrow Y$ defined by
\[ \pi(x_1, x_2, x'', x_n) := (\rho, x''), \quad \rho := \sqrt{x_1^2+x_2^2}. \]
Then $\pi(\Phi_k(\theta, \rho, x'')) = (\rho, x'')$. Define $\nu$ as the pushforward of $\sigma_{k,R}$ under $\pi$:
\begin{equation}\label{definition:measure-nu-paraboloid}
\nu(A) := \sigma_{k,R}(\pi^{-1}(A)), \quad \text{for every Borel set } A \subset Y.
\end{equation}

The projection collapses each rotational orbit in the $(x_1, x_2)$-plane to a single point in $Y$. By Rokhlin's disintegration theorem (which allows us to write a measure as an integral over conditional measures along fibers), there exists a family of conditional probability measures $\{\mu_{(\rho, x'')}\}_{(\rho, x'') \in Y}$ on $\mathbb{T}$ such that for every bounded measurable $g : \mathbb{T} \times Y \rightarrow \mathbb{R}$,
\[ \int_{\mathcal{S}_{k,R}} g(\theta(x), \pi(x)) \, d\sigma_{k,R}(x) = \int_Y \int_{\mathbb{T}} g(\theta, \rho, x'') \, d\mu_{(\rho, x'')}(\theta) \, d\nu(\rho, x''), \]
where $\theta(x)$ is any measurable choice with $x = \Phi_k(\theta(x), \pi(x))$.

For each fixed $(\rho,x'')\in Y$, the fibre $\pi^{-1}(\{\rho,x''\})$ is
the circle
\[
\big\{\Phi_k(\theta,\rho,x'') : \theta\in \mathbb{T}\big\}
 = \big\{(\rho\cos 2\pi\theta,\rho\sin 2\pi\theta,
          x'',(\rho^2+|x''|^2)^{k/2}) : \theta\in \mathbb{T}\big\},
\]
on which the action $\{R_\alpha\}$ is transitive via
$\theta\mapsto\theta+\alpha$ and preserves $\sigma_{k,R}$. Hence, for $\nu$–almost every $(\rho,x'')$, the conditional measure $\mu_{(\rho,x'')}$ is invariant under all rotations on $\mathbb{T}$, and thus
\[
\mu_{(\rho,x'')} = \mu_1,
\]
the normalized Lebesgue measure on $\mathbb{T}$.

Taking $g(\theta,\rho,x''):=f(\Phi_k(\theta,\rho,x''))$ in the
disintegration identity and using that $\Phi_k$ is bijective from $\mathbb{T} \times Y$ onto $\mathcal{S}_{k,R}\setminus N$ with
$\sigma_{k,R}(N)=0$, we obtain, for every bounded measurable
$f:\mathcal{S}_{k,R}\to\mathbb{R}$,
\[
\int_{\mathcal{S}_{k,R}} f(x)\,d\sigma_{k,R}(x)
 = \int_\mathbb{T} \int_Y f(\Phi_k(\theta,y))\,d\nu(y)\,d\mu_1(\theta),
\]
which is exactly the product disintegration condition~(ii) in Theorem~\ref{thm:general-circle-bundle}.

\medskip\noindent
Having verified all hypotheses of Theorem~\ref{thm:general-circle-bundle} with
\[
X=\mathcal{S}_{k,R},\quad 
\mu=\sigma_{k,R},\quad
Y \text{ given in } \eqref{definition-Y},\quad
\nu \text{ defined in \eqref{definition:measure-nu-paraboloid}},\quad
N=\{x\in \mathcal{S}_{k,R}: x_1=x_2=0\},
\]
and with $\Phi$ given in \eqref{definition:Phi_k-paraboloid}, we may now apply Theorem~\ref{thm:general-circle-bundle}.

Let $\{E_i^\mathbb{T}\}_{i=1}^r$ be the partition of $\mathbb{T}$ with the Raimi property from Theorem~\ref{thm:KKT-circle}, and define
\[
E_1^{\mathcal{S}_{k,R}} := \Phi_k(E_1^\mathbb{T}\times Y)\cup N,\qquad
E_i^{\mathcal{S}_{k,R}} := \Phi_k(E_i^\mathbb{T} \times Y),\quad 2\le i\le r.
\]
Theorem~\ref{thm:general-circle-bundle} then yields the asserted property of the partition $\{E_i^{\mathcal{S}_{k,R}}\}_{i=1}^r$.


\section{Proof of Theorem \ref{thm:cylinder-Raimi} (Cylindrical surface)}\label{section-cylinder}

We apply Theorem~\ref{thm:general-circle-bundle} with
\[
X = \mathcal{C}_{R,\Omega}, \qquad \mu = \mu_{R,\Omega}.
\]
By construction, the maps $\{R_\theta\}$ act on $\mathcal{C}_{R,\Omega}$ by
rotations in the $(x_1,x_2)$–plane; this action is measurable and preserves
the normalized surface measure $\mu_{R,\Omega}$.

\medskip
In the cylindrical case, we identify the base space with the natural
coordinate set
\[
Y=\Omega:=\{(x'',x_n)\in\mathbb{R}^{n-3}\times\mathbb{R}:\ (x'',x_n)
\ \text{satisfies the defining constraints}\},
\]
so that the parametrization is written in the natural coordinates
$\Phi(\theta,x'',x_n)$.

\medskip\noindent
\emph{Choice of $Y$, $N$, and $\Phi$.}
Since there is no singular axis in this case, we take
\[
N := \varnothing.
\]
Define the parametrization
\begin{align}\label{definition:Phi-cylinder}
    \Phi: \mathbb{T} \times Y \to \mathbb{R}^n, \qquad
    \Phi(\theta,x'',x_n)
    := (R\cos 2\pi\theta,\ R\sin 2\pi\theta,\ x'',x_n).
\end{align}
A direct computation shows that
$\Phi:\mathbb{T}\times Y \to \mathcal{C}_{R,\Omega}$ is a measurable bijection.


\emph{Verification of equivariance (assumption (i)).}
Since $R_\alpha$ only rotates the $(x_1,x_2)$–coordinates, we obtain
\begin{align*}
R_\alpha(\Phi(\theta, x'', x_n)) 
 &= R_\alpha\big(R\cos 2\pi\theta,\ R\sin 2\pi\theta,\ x'', x_n\big)\\
 &= \big(R\cos 2\pi(\theta+\alpha),\ R\sin 2\pi(\theta+\alpha),\ x'', x_n\big)\\
 &= \Phi(\theta+\alpha, x'', x_n).
\end{align*}
Thus $\Phi$ intertwines the circle action, establishing condition~(i) of Theorem~\ref{thm:general-circle-bundle}.

\medskip\noindent
\emph{Product disintegration (assumption (ii)).}
Consider the projection
\[
\pi : \mathcal{C}_{R, \Omega}\setminus N \to Y,\qquad
\pi(x_1,x_2,x'',x_n) := (x'', x_n).
\]
Then $\pi(\Phi(\theta,x'',x_n) ) = (x'',x_n)$. 

Define $\nu$ as the pushforward of $\mu_{\text{\tiny $R,\Omega$}}$ under $\pi$, namely
\begin{align}\label{definition:measure-cylinder}
    \nu(A):=\mu_{\text{\tiny $R,\Omega$}}\big(\pi^{-1}(A)\big),
\qquad \text{ for every Borel set } ~~ A\subset Y.
\end{align}
The projection $\pi$ collapses each rotational orbit in the $(x_1,x_2)$–plane to a single point in $Y$.  
By Rokhlin’s disintegration theorem, there exists a family of 
probability measures $\{\mu_{(x'', x_n)}\}_{(x'', x_n)\in Y}$ on $\mathbb{T}$ such that for every bounded measurable function $g:\mathbb{T} \times Y\to\mathbb{R}$,
\begin{equation}\label{eq:cyl-disintegration}
\int_{\mathcal{C}_{R,\Omega}} g(\theta(x),\pi(x))\,d \mu_{\text{\tiny $R,\Omega$}}(x)
= \int_Y\int_\mathbb{T} g(\theta, x'', x_n)\,d\mu_{(x'', x_n)}(\theta)\,d\nu(x'', x_n),
\end{equation}
where $\theta(x)$ is any measurable choice with $x=\Phi(\theta(x),\pi(x))$.

For each fixed $(x'', x_n) \in Y$, the fibre $\pi^{-1}(\{(x'', x_n)\})$ is the circle
\[
\{(x_1,x_2, x'', x_n) : x_1^2+x_2^2=R^2\},
\]
which is parametrized by $\theta\mapsto\Phi(\theta, x'', x_n)$. The action $\{R_\alpha\}$ of $\mathbb{T}$ on $X$ restricts to rotations on each such fibre and preserves $\mu_{\text{\tiny $R,\Omega$}}$. Therefore, for $\nu$–almost every $(x'', x_n)\in Y$, the conditional measure $\mu_{(x'', x_n)}$ must be invariant under all translations $\theta\mapsto\theta+\alpha$ on $\mathbb{T}$. The only probability measure on $\mathbb{T}$ with this invariance property is the normalized Lebesgue measure $\mu_1$, and hence
\[
\mu_{(x'', x_n)} = \mu_1\qquad\text{for $\nu$–a.e. } (x'', x_n)\in Y.
\]

Taking $g(\theta, x'', x_n):=f(\Phi(\theta, x'', x_n))$ in
\eqref{eq:cyl-disintegration}, we obtain
\[
\int_{\mathcal{C}_{R,\Omega}} f(x)\,d\mu_{\text{\tiny $R,\Omega$}}(x)
= \int_Y\int_\mathbb{T} f(\Phi(\theta, x'', x_n))\,d\mu_1(\theta)\,d\nu(x'', x_n)
= \int_\mathbb{T} \int_Y f(\Phi(\theta, x'', x_n))\,d\nu(x'', x_n)\,d\mu_1(\theta)
\]
for every bounded measurable $f:\mathcal{C}_{R,\Omega}\to\mathbb{R}$.
This is precisely condition~(ii) in Theorem~\ref{thm:general-circle-bundle}.

With all hypotheses verified, Theorem~\ref{thm:general-circle-bundle} now applies with
\[
X=\mathcal{C}_{R,\Omega},\quad \mu=\mu_{\text{\tiny $R,\Omega$}},\quad
Y=\Omega,\quad
\nu \text{ defined in \eqref{definition:measure-cylinder}},\quad
N= \varnothing,\quad
\Phi \text{ given in \eqref{definition:Phi-cylinder}},
\]
to obtain a measurable partition
\[
\{E_i^{\mathcal{C}_{R,\Omega}}\}_{i=1}^r
   =\{\Phi(E_i^\mathbb{T} \times Y)\}_{i=1}^r
\]
of $\mathcal{C}_{R,\Omega}$, where $\{E_i^\mathbb{T}\}_{i=1}^r$ is the measurable partition of $\mathbb{T}$ provided by Theorem~\ref{thm:KKT-circle}, with the desired property. 
This completes the proof.

{\bf Acknowledgement.} 
The author would like to express sincere gratitude to Hunseok Kang and Doowon Koh for their valuable comments, which greatly improved the quality of this paper. The author also thanks the Vietnam Institute for Advanced Study in Mathematics (VIASM) for its warm hospitality and excellent working environment, and Dang Thuy Trang for her assistance with the figure.


\begin{thebibliography}{99}


\bibitem{Bergelson2} V. Bergelson and B. Weiss, \textit{Translation Properties of Sets of Positive Upper Density}, Proceedings of the American Mathematical Society, \textbf{94}(3) (1985), 371--376.

\bibitem{NH} N. Hegyv\'ari, 
\textit{On intersecting properties of partitions of integers},
Combinatorics, Probability and Computing, \textbf{14}(3) (2005), 319--323.

\bibitem{HPP25}
N.~Hegy\-v\'{a}ri, J.~Pach, and T. Pham,
\textit{Polynomial extensions of Raimi's theorem}, 
\href{https://arxiv.org/abs/2511.06650}{arXiv:2511.06650 [math.CO]}, (2025)

\bibitem{HM79} N. Hindman, \textit{Ultrafilters and combinatorial number theory}, Number theory, Carbondale 1979 (Proc. Southern Illinois Conf., Southern Illinois Univ., Carbondale, Ill., 1979),  pp. 119--184, Lecture Notes in Math., 751, Springer, Berlin, 1979.

\bibitem{KKT25} H. Kang, D. Koh, and D. T. Tran, \textit{Raimi's theorem for the $n$-dimensional torus}, 
\href{https://arxiv.org/abs/2512.00935}
{arXiv:2512.00935 [math.CO]}, (2025)

\bibitem{Raimi}  R. Raimi, \textit{Translation properties of finite partitions of the positive integers}, Fundamenta Mathematicae, \textbf{61}(3) (1967), 253--256.
\end{thebibliography}
\end{document}